\documentclass[12pt]{amsart}   \usepackage{amssymb,latexsym}
\usepackage{amsfonts,mathrsfs}
\usepackage[margin=1.4in]{geometry}

\newtheorem{theorem}{Theorem}[section]
\newtheorem{lemma}[theorem]{Lemma}

\theoremstyle{remark}
\newtheorem{remark}{Remark}[section]

\begin{document}
\title{ Eigenfunctions for quasi-laplacian}
\keywords{Quasi-Laplacian, singularity, eigenfunction}
\thanks{\noindent \textbf{MR(2010)Subject Classification}   47F05 58C40}

\author{min chen}
\address[Corresponding author] {University of Science and Technology of China, No.96, JinZhai Road Baohe District,Hefei,Anhui, 230026,P.R.China.} 
\email{cmcm@mail.ustc.edu.cn}
\thanks{The research is supported by the National Nature Science Foudation of China  No. 11721101 No. 11526212 }
\begin{abstract}
To study the regularity of heat flow, Lin-Wang[1] introduced the quasi-harmonic sphere, which is a harmonic map from $M=(\mathbb{R}^m,e^{-\frac{|x|^2}{2(m-2)}}ds_0^2)$ to $N$ with finite energy. Here $ds_0^2$ is Euclidean metric in $\mathbb{R}^m$. Ding-Zhao [2] showed that if the target is a sphere, any equivariant  quasi-harmonic spheres is discontinuous at infinity. The metric $g=e^{-\frac{|x|^2}{2(m-2)}}ds_0^2$ is quite singular at infinity and it is not complete. In this paper , we mainly study the eigenfunction of Quasi-Laplacian $\Delta_g=e^{\frac{|x|^2}{2(m-2)}} ( \Delta_{g_0} - \nabla_{g_0}h\cdot \nabla_{g_0}) =e^{\frac{|x|^2}{2(m-2)}} \Delta_h$ for $h=\frac{|x|^2}{4}$. In particular, we show that non-constant eigenfunctions of $\Delta_g$ must be discontinuous at infinity and non-constant eigenfunctions of drifted Laplacian $\Delta_h=\Delta_{g_0} - \nabla_{g_0} h\cdot \nabla_{g_0}$ is also discontinuous at infinity.
\end{abstract}
\maketitle

\numberwithin{equation}{section}
\section{Introduction}
\setcounter{equation}{0}
 If a heat flow $u(x,t)$ from M to N blows up at a finite time,  Lin-Wang \cite{1}  and Struwe \cite{10} proved that there exists a harmonic map from $R^m$ to $N$ with the conformal metric  $g=e^{-\frac{|x|^2}{2(m-2)}}ds_0^2$ with finite energy of $\omega$ w.r.t this metric 
\[E_g(\omega)=\int_{\mathbb{R}^m}|\nabla \omega|^2 e^{-\frac{|x|^2}{4}}dx<\infty.\]

In \cite{2}, Ding-Zhao showed that equivariant quasi-harmonic spheres are discontinuous at infinity if $N=S^m$. So the behavior of quasi-harmonic spheres is quite different from that of harmonic spheres. Since the metric is very singular at infinity, and it is not complete, we are interested in the analytic property of the Laplace operator of $(\mathbb{R}^m,g)$  which we call Quasi-Laplacian, denoted by $\Delta_g=e^{\frac{|x|^2}{2(m-2)}} \bigg( \Delta_{g_0} - \nabla_{g_0}h\cdot \nabla_{g_0}\bigg) $, where $h=\frac{|x|^2}{4}.$ 
Our first result is the compact embedding theorem. 
\begin{theorem}\label {thm:1.1}
Let $\bar{M}=(\mathbb{R}^m,g,dV_g)$, the embedding operator $H^1_0(\bar M)\hookrightarrow L^2(\bar{M})$ is compact.
\end{theorem}

 Cheng and Zhou \cite{4} used the result of Hein- Naber \cite{5} to obtain the Lichnerowicz type theorem for the drifted Laplacian $\Delta_h=\Delta_{g_0} - \nabla_{g_0}h\cdot \nabla_{g_0}$ as follows.
~\\

\textbf{Theorem}
(Bakry-$\acute{E}$mery-Mogan-Hein-Naber) Let $(M^n, g ,e^{-f})$ be a complete smooth metric measure space with 
$Ric_f \ge \frac{a}{2}g$ for some constant $a>0.$ Then
 
(1) the spectrum of $\Delta_f$  for $M$ is discrete.

(2) the first nonzero eigenvalue, denoted by $\lambda_1(\Delta_f),$ of $\Delta_f$ for $M$ is the spectrum gap of $\Delta_f$ and satisfies 
\[\lambda_1(\Delta_f)\ge \frac{a}{2},\]
~\\
here $\Delta_f=\Delta-\nabla f\cdot\nabla$. 

So we wonder if the spectrum of $\Delta_g$ is also discrete. According to Theorem 1.1, we know the compact embedding theorem on $\mathbb{R}^m$ with metric $g$ still holds even though the metric is not complete. Combined with Theorem 10.6 in [6], we obtain the following result.
\begin{theorem}\label {thm:1.2}
The spectrum of $\Delta_g$ in $\mathbb{R}^m$ with the conformal metric $g=e^{-\frac{|x|^2}{2(m-2)}}g_0$ is discrete.
\end{theorem}
The discreteness of spectrum of $\Delta_g$ guarantees the existence  of eigenfunctions of $\Delta_g. $ The metric is singular at infinity, so we are interested in continuity of eigenfunctions at infinity. And we prove that:
 \begin{theorem}\label {thm:1.3}
Let $u$ be a non-constant eigenfunction of the quasi-Laplacian $\Delta_g$ corresponding to any eigenvalue $\lambda$, then $u$ must be discontinuous at $\infty.$
\end{theorem}
 And we  find that the eigenfunction of $\Delta_h$  has the same property.
\begin{theorem}\label {thm:1.4}
Let $u$ be a non-constant eigenfunction of  the drifted Laplacian $\Delta_h$ corresponding to an eigenvalue $\lambda$, then $u$ must be discontinuous at $\infty.$
\end{theorem}
 According to the Proposition 2.1 in \cite{8},  a logarithmic Sobolev inequality (2.2) implies a Poincar$\acute{e}$ inequality. Then we get the existence of a global solution of $\Delta_gu=f$ in $\mathbb{R}^m$ in the last section. 
\hspace{0.4cm}

\section{compact embedding theorem and discrete spectrum}
\setcounter{equation}{0}
\medskip
We first recall some  facts in measure theory. 
A subset $K$ of $L^1(\mu)$ is called uniformly integrable if given $\epsilon >0,$ there is a $\delta>0$ so that $\text{sup}\{\int_E|f|d\mu:f\in K\}<\epsilon$ whenever $\mu(E)<\delta.$ It is known that
\begin{lemma}
(De La Vall$\acute{e}$e Poussin theorem, cf \cite{8}) Under the above notation, a subset $K$ of $L^1(\mu)$ is uniformly integrable if and only if there exists a non-negative convex function $Q$ with $\lim\limits_{t\to\infty}\frac{Q(t)}{t}=\infty$ so that 
\[\text{sup}\bigg\{\int_{\Omega}Q(|f|)d\mu: f\in K\bigg\}<\infty.\]
\end{lemma}
In local coordinates,
\begin{align*}
& |\nabla_g u|^2_g =g^{ij} \frac{\partial u^{\alpha}}{\partial x^i}\frac{\partial u^{\alpha}}{\partial x^j}=e^{\frac{|x|^2}{2(m-2)}}|\nabla_{g_0} u|^2_{g_0};\\
& dV_g =\sqrt{detg}dV_{g_0} =e^{-\frac{m}{4(m-2)}|x|^2}dV_{g_0} .\\
\end{align*}
And
\begin{align*}
\Delta_g &= \frac{1}{\sqrt{detg}} \frac{\partial}{\partial x^i} \big(\sqrt{detg}\ g^{ij} \frac{\partial}{\partial x^j}\big) \\
&= e^{\frac{m|x|^2}{4(m-2)}} \frac{\partial}{\partial x^i} \bigg( e^{-\frac{m|x|^2}{4(m-2)}} \ e^{\frac{|x|^2}{2(m-2)}} \ \frac{\partial}{\partial x^i}\bigg) \\
&= e^{\frac{m|x|^2}{4(m-2)}} \bigg( \frac{\partial^2}{\partial x_i^2} - \frac{x_i}{2} \frac{\partial}{\partial x_i}\bigg) \\
&= e^{\frac{|x|^2}{2(m-2)}} \bigg( \Delta_{g_0} - \nabla_{g_0}h\cdot \nabla_{g_0}\bigg) \\
&= e^{\frac{|x|^2}{2(m-2)}}(\Delta_{g_0})_h,
\end{align*}
where $h=\frac{|x|^2}{4}$. 
Note that 
\begin{equation}
\begin{split}
&\int_{\mathbb{R}^m} u^2dV_g=\int_{\mathbb{R}^m}u^2 e^{-\frac{|x|^2}{4(m-2)}m}dV_{g_0}; \\
&\int_{\mathbb{R}^m}|\nabla_g u|_g dV_g = \int_{\mathbb{R}^m}|\nabla_{g_0} u|_{g_0}^2 e^{-\frac{|x|^2}{4}}dV_{g_0}.\\
\end{split}
\end{equation}
The $H^1_0({\mathbb{R}^m}, g, dV_g)$-norm of $u$ with the metric $g$ can be view as the sum of  the $L^2({\mathbb{R}^m}, g_0,e^{-\frac{|x|^2}{4(m-2)}m} dV_{g_0})$-norm of $u$ and the $L^2({\mathbb{R}^m}, g_0,e^{-\frac{|x|^2}{4}}dV_{g_0})$-norm of $\nabla_{g_0}u$. We can use the result  which is already known in the complete smooth metric measure space $({\mathbb{R}^m}, g_0,e^{-f}dV_{g_0})$. 

Considering that $\mathbb{R}^m$ is not compact, we need to do more work. To deal with this non-compact case, we construct a compact exhaustion $\{D_i\}$ of  $\mathbb{R}^m$ with $C^1$ boundary.   We know $\mu(\bar{M})=\int_{\mathbb{R}^m}dV_g=\int_{\mathbb{R}^m}e^{-\frac{m}{4(m-2)}|x|^2}dV_{g_0}$. Assume $d\mu=e^{-\frac{m|x|^2}{4(m-2)}}dV_{g_0}$. Then we can use the result  that $H^1_0(D_i,g_0,d\mu) \subset L^2(D_i,g_0,d\mu)$ is compact. 

And in 1985, Bakery-Emery \cite{3} showed that if $(M,g,e^{-f}dV)$ has $\text{Ric} f\ge \frac{a}{2}g$ for some constant $a\ge 0$ and finite weighted volume $\int_Me^{-f}dV,$ then the following logarithmic Sobolev inequality holds:
\begin{equation}
\int_Mu^2\text{log}(u^2)e^{-f}dv\le \frac{4}{a}\int_M|\nabla u|^2e^{-f}dv,
\end{equation}
for all smooth function $u$. Here $\text{Ric}_f:=\text{Ric}+\nabla^2f$ . 
\begin{remark}
the logarithmic Sobolev inequality (2.2) holds for all $u\in C^\infty_0(M)$, then it holds for all $u\in H_0^1(M,g_0,e^{-f}dv).$\\
In fact, for $u\in H_0^1(M,g_0,e^{-f}dv),$ there exits a sequence $\{u_k\}$, $u_k\in C^\infty_0(M)$ and $u_k\rightarrow u$ in $H_0^1(M,g_0,e^{-f}dv)$. Since $u_k\rightarrow u$ in $L^2(M,g_0,e^{-f}dv)$, there is a subsequence of $u_k$, still denoted by $u_k$, and $u_k$ a.e converges to $u$. Then we will have \\
\begin{align*}
0&\le \int_Mu^2\text{log}u^2e^{-f}dv\\
&\le\text{ lim inf}\int_Mu_k^2 \text{log}u_k^2e^{-f}dv\\
&\le \text{ lim inf} (\frac{4}{a}\int_M|\nabla u_k|^2e^{-f}dv)\\
&=\frac{4}{a}\int_M|\nabla u|^2e^{-f}dv.
\end{align*}
Hence (2.2) holds for $u\in H_0^1(M,g_0,e^{-f}dv).$
\end{remark}
 Combined with the logarithmic Sobolev inequality with which we obtain the uniformly integrability, we can use Vitali convergence theorem 
 to deduce compact  embedding of $H^1_0(\bar{M})$ in $ L^2(\bar{M}).$
\begin{proof}[Proof of Theorem \ref{thm:1.1}]
It is known that identical map $H^1_0(\bar M)\rightarrow L^2(\bar M)$ is a embedding. So it suffices to prove that any sequence of $\{u_k\}^{\infty}_{k=1}$ bounded in $H^1_0(\bar M)$ has subsequence converging in $L^2(\bar M)$ to a function $u\in L^2(\bar M).$ We will use Vitali convergence theorem to deduce that $\int_{\mathbb{R}^m}|u_k-u|^2dV_g\rightarrow 0.$ Since $\mu(\bar M)< +\infty,$ it suffices for us to prove that \\
(1) $\{u_k\}^{\infty}_{k=1}$ converges in measure to $u$;\\
(2) $\{u_k\}^{\infty}_{k=1}$ is uniformly integrable.\\
Let $\{D_i\}$ be an compact exhaustion of $\mathbb{R}^n$, with $C^1$ boundary $\partial D_i$. For $\Omega_i=\{D_i,g,dV_g\}$ and any $\{u_k\}$ bounded in $H^1_0(\Omega_i),$
\medskip

\begin{align*}
\|u_k\|_{H^1(D_i,g_0,d\mu)}&=\int_{D_i}u_k^2e^{-\frac{|x|^2}{4(m-2)}m}dV_{g_0}+ \int_{D_i}|\nabla_{g_0} u_k|_{g_0}^2 e^{-\frac{m|x|^2}{4(m-2)}}dV_{g_0} \\
&\leq \int_{D_i}u_k^2 e^{-\frac{|x|^2}{4(m-2)}m}dV_{g_0} + \int_{D_i}|\nabla_{g_0} u_k|_{g_0}^2 e^{\frac{|x|^2}{2(m-2)}} e^{-\frac{m|x|^2}{4(m-2)}}dV_{g_0} \\
&= \int_{D_i}u_k^2dV_g + \int_{D_i}|\nabla_g u_k|_g ^2dV_g \\
&=\|u_k\|_{H^1_0(\Omega_i)} \leq\|u_k\|_{H^1_0(\bar M)} \leq C.
\end{align*}
It implies $\{u_k\}$ bounded in $H^1_0(D_i,g_0,d\mu)$ and it is known that $H^1_0(D_i,g_0,d\mu) \subset L^2(D_i,g_0,d\mu)$ is compact. So the sequence $\{u_k\}$ restrict to $D_i$ has subsequence converging in $L^2(D_i,g_0,d\mu)$. 
Considering that $||u_k||_{L^2(D_i,g,dV_g)}=||u_k||_{L^2(D_i,g_0,d\mu)},$ we obtain that the sequence $\{u_k\}$ restrict to $D_i$ has subsequence converging in $L^2(D_i,g,dV_g).$ Note that an $L^2$ convergent sequence has an a.e convergent subsequence. By passing to a diagonal subsequence, there exists a subsequence of $\{u_k\}$ still denoted by  $\{u_k\}$ and a function $u$ defined on $\mathbb{R}^m$ so that ${u_k}$ a.e converges to $u$ on each $D_i$ and hence on $\mathbb{R}^m$. By fatou's lemma,

\begin{align*}
\int_{\mathbb{R}^m}|u|_g^2dV_g &= \int_{\mathbb{R}^m}{}|u|_{g_0}^2 e^{-\frac{m}{4(m-2)}|x|^2}dV_{g_0} \leq \liminf \int_{\mathbb{R}^m}{}|u_k|_{g_0}^2 e^{-\frac{m}{4(m-2)}|x|^2} dV_{g_0} \\
&=\liminf \int_{\mathbb{R}^m} |u_k|_g^2 dV_g.
\end{align*}
We can get $u\in L^2(\bar{M})$ and Condition(1) holds. \\
On the other hand, by the hypothesis of the theorem, if we take $M=\mathbb{R}^m$ and $f=\frac{m}{4(m-2)}|x|^2$, then $\int_{\mathbb{R}^m} e^{-f}dV_{g_0}< +\infty$ and $\text{Ric}f=\frac{m}{2(m-2)}g_0.$  The logarithmic Sobolev inequality (2.2) holds for $H_0^1(\mathbb{
R}^m,g_0,e^{-f}dV_{g_0})$. \\
\begin{align*}
\int_{\mathbb{R}^m}u_k^2 \text{log} u_k^2 d\mu&=\int_{\mathbb{R}^m}u_k^2 \text{log} u_k^2 e^{-\frac{m}{4(m-2)}|x|^2}dV_{g_0}\\ &\le \frac{4(m-2)}{m}\int_{\mathbb{R}^m}|\nabla_{g_0} u_k|_{g_0}^2e^{-\frac{m}{4(m-2)}|x|^2}dV_{g_0} \\
&\le \frac{4(m-2)}{m}\int_{\mathbb{R}^m}|\nabla_{g} u_k|_{g}^2e^{-\frac{m}{4(m-2)}|x|^2}dV_{g_0} \\
&= \frac{4(m-2)}{m}\int_{\mathbb{R}^m}|\nabla_{g} u_k|_{g}^2dV_{g} \\
&\le ||u_k||_{H^1_0(\bar{M})}.
\end{align*}
With the boundedness of $H^1_0(\bar{M})$-norm of $u_k$, it implies that there exists a constant $\bar{C}$ satisfying 
\[\int_{\mathbb{R}^m}u_k^2 \text{log} u_k^2 d\mu\le \bar{C}.\]
Take $Q(t)=t \log t$, one can see that $Q(t)$ and $\{u_k^2\}$ satisfy the conditions of Lemma 2.1 in \cite{8} and thus $\{u_k^2\}$ is uniformly integrable. Condition (2) holds. Therefore the embedding operator $H^1_0(\bar M)\hookrightarrow L^2(\bar M)$ is compact.
\end{proof}

Applying the following result for weighted manifold, we can immediately obtain Theorem 1.2.
\begin{theorem}
\cite{6} Let $(M,g,\mu)$ be a weighted manifold. Then the following conditions are equivalent.

(a) The spectrum of $M$ is discrete.

(b) The embedding operator $W^1_0(M)\hookrightarrow L^2(M)$ is compact.

(c) The resolvlent $R_\alpha=(\mathscr{L}+\alpha \text{i}d)^{-1}$ is a compact operator in $L^2(M)$, for some $\alpha>0$
\end{theorem}


\section{The discontinunity of eigenfunction at infinity}
Assume that $u$ is non-constant eigenfunction of $\Delta_g$ corresponding to an eigenvalue $\lambda,$ i.e.,
\begin{equation}
\Delta_g u= -\lambda u.
\end{equation}
We rewrite it in the following form
\begin{equation}
\Delta u-( \nabla h,\nabla u)= -\lambda u e^{-\frac{|x|^2}{2(m-2)}}.
\end{equation}
We know that the Euclidean metric of $\mathbb{R}^m$ can be written in spherical coordinates $(r,\theta)$ as 
\[ds^2=dr^2+r^2d\theta^2,\]
where $d\theta^2$ is the standard metric on $S^{m-1}$.
Then
\begin{align*}
\Delta &= \frac{1}{\sqrt{detg}} \frac{\partial}{\partial r} \big(\sqrt{detg}\frac{\partial}{\partial r }\big)+\frac{1}{\sqrt{detg}} \frac{\partial}{\partial \theta_i} \big(\sqrt{detg}\ r^{-2}g_\theta^{ii} \frac{\partial}{\partial \theta_i}\big) \\
&= \frac{1}{r^{m-1}}\frac{\partial}{\partial r}(r^{m-1}\frac{\partial}{\partial r})+\frac{1}{r^2}\frac{1}{\sqrt{g_\theta}}\sum \frac{\partial}{\partial\theta_i}(g^{ii}_\theta\sqrt{g_\theta}\frac{\partial}{\partial{\theta_i}})\\
&= \frac{\partial^2}{\partial r^2}+\frac{m-1}{r}\frac{\partial}{\partial r}+\frac{1}{r^2} \Delta_\theta\\
&= \Delta_r+\frac{1}{r^2}\Delta_\theta,
\end{align*}
where $\Delta_\theta$ is the Laplacian on the standard $S^{m-1}.$
It is clear that  
\[\nabla h\cdot \nabla=\frac{r}{2}\frac{\partial}{\partial r}.\]
It follows from (3.2) that 

\[u_{rr}+\frac{m-1}{r}u_r+\frac{1}{r^2}\Delta_\theta u-\frac{r}{2}\frac{\partial u}{\partial r}=-\lambda e^{-\frac{r^2}{2(m-2)}}u.\]

Let $\varphi_k$ be the orthonormal basis on $L^2(S^{m-1})$ corresponding to the eigenvalues \\
\[0=\lambda_0<\lambda_1\le \lambda_2\le \cdots \le \lambda_k\rightarrow \infty \]
and satisfying 
\[\Delta_\theta \varphi_k=\lambda_k \varphi_k.\]
Also let $\langle \cdot,\cdot \rangle$ denote $L^2$ inner product of $L^2(S^{m-1}).$ Then we have 

\begin{align*}
\langle \Delta_r u,\varphi_k \rangle&=\Delta_r \langle u, \varphi_k \rangle,\\
\langle \Delta_\theta u,\varphi_k \rangle&= \langle u, \Delta_\theta \varphi_k \rangle =-\lambda_k\langle u,\varphi_k\rangle,\\
\langle \frac{\partial u}{\partial r},\varphi_k\rangle&=\frac{\partial \langle u,\varphi_k\rangle}{\partial r}.
\end{align*}
Let $f_k(r)=\langle u(r,\cdot),\varphi_k\rangle$ for $k\ge1$ and we denote $f=f_k.$
Then  we can reduce the PDE $ \Delta_g u= -\lambda u$ to an ODE 
\begin{equation}
f_{rr}+(\frac{m-1}{r}-\frac{r}{2})f_r=(-\lambda e^{-\frac{r^2}{2(m-2)}}+r^{-2}\lambda_k)f.
\end{equation}

In the following theorem (Theorem 3.1), under the assumption that $u$ is continuous at infinity, we will have $u=f_0(r)$, which means that $u$ is radial symmetry. Then we use maximum principle to get a contradiction, so that we prove Theorem 1.3.
\begin{theorem}
Let $u$ be an eigenfunction of $\Delta_g$ corresponding to an eigenvalue $\lambda$. If $u$ is continuous at infinity, then $f_k=\langle u,\varphi_k\rangle\equiv 0$ for $k\ge1.$ In other word, $u$ is radial symmetry.
\end{theorem}
\begin{proof}
We rewrite (3.3) in divergence form
\begin{equation}
(f^{\prime} r^{m-1}e^{-\frac{r^2}{4}})^\prime=r^{m-1}e^{-\frac{r^2}{4}}(-\lambda e^{-\frac{r^2}{2(m-2)}}+r^{-2}\lambda_k)f.\end{equation}
Since $u$ is continuous at infinity,  $\lim\limits_{r\to \infty}{ u(r,\cdot)=A}$, we have
\begin{center}
$f_k(+\infty)=\lim\limits_{r\to +\infty}{\langle u(r,\cdot),\varphi_k\rangle}=\langle \lim\limits_{r\to +\infty} {u(r,\cdot),\varphi_k\rangle}=\langle A,\varphi_k\rangle=0,$ \quad for $(k\ge 1).$
\end{center}
It is clear that 
\begin{center}
$-\lambda e^{-\frac{r^2}{2(m-2)}}+r^{-2}\lambda_k=e^{-\frac{r^2}{2(m-2)}}(\frac{e^{\frac{r^2}{2(m-2)}}}{r^2}\lambda_k-\lambda)>0$  as $r>>1$.
\end{center}
Then by maximum principle, we can see from (3.4) that $f$ cannot have positive maximum or negative minimum unless $f\equiv 0$. If $f$ is not identically zero, then we can show that
$f$ must be monotonic for $r>>1$. Without loss of generality, we assume that $f^{\prime}\le 0$ (for $r>>1)$, hence $f(r)\ge 0.$ 

Integrating from $r$ to $+\infty,$ we  will have 
\begin{align*}
0-f^{\prime} r^{m-1}e^{-\frac{r^2}{4}}&=\int^{+\infty}_{r} r^{m-1}e^{-\frac{r^2}{4}}(-\lambda e^{-\frac{r^2}{2(m-2)}}+r^{-2}\lambda_k)f dr\\
&\le \int^{+\infty}_{r} r^{m-3}e^{-\frac{r^2}{4}}\lambda_kf dr \\
&\le \int^{+\infty}_{r} r^{m-3}e^{-\frac{r^2}{4}}\lambda_kdr f.\\
\end{align*}
Then
\begin{align*}
-\frac{f^{\prime}}{f}&\le  r^{1-m}e^{\frac{r^2}{4}}\int^{+\infty}_{r}r^{m-3}e^{-\frac{r^2}{4}}\lambda_kdr \\
&=\lambda_k\frac{\int^{+\infty}_{r}r^{m-3}e^{-\frac{r^2}{4}}}{r^{m-1}e^{-\frac{r^2}{4}}}\\
&=\lambda_k\frac{1}{r^3}(2+o(1)).
\end{align*}
It is equivalent to
\[(fe^{-\frac{\lambda_k}{r^{2}}(1+o(1))})^\prime\ge 0.\]
Then we have 
\begin{center}
$f\ge c_0e^{\frac{\lambda_k}{r^2}(1+o(1))}\rightarrow c_0$ as $r\rightarrow +\infty,$
\end{center}
which contradicts the fact that $f_k(+\infty)=0.$ Hence we must have $f=f_k\equiv 0$ for $k\ge 1$.
It follows that 
\[u=\sum_k\langle u,\varphi_k\rangle\varphi_k=\sum_kf_k\varphi_k=f_0\varphi_0=\int_{S^{n-1}}ud\theta=f_0(r).\]
So $u$ is radial symmetry.
\end{proof}

\begin{proof}[Proof of Theorem \ref{thm:1.3}]
If the conclusion is not true, we assume that $\lim\limits_{r\to \infty}{ u(r,\cdot)=A}.$
From the result of Theorem 3.1, we know that $u$ is radial symmetry and $u=f_0(r)$.
So it satisfies following equation
\[u^{\prime\prime}+(\frac{m-1}{r}-\frac{r}{2})u^{\prime}=-\lambda e^{-\frac{r^2}{2(m-2)}}u.\]
Let $L=\frac{d^2}{dr^2}+(\frac{m-1}{r}-\frac{r}{2})\frac{d}{dr}$ and we choose $y=e^{2r^2},$ then we have

\[L(y-u)=e^{2r^2}(14r^2+4m)+\lambda e^{-\frac{r^2}{2(m-2)}}u\]
Since 
\begin{center}
$e^{-\frac{r^2}{2(m-2)}}u\rightarrow 0 $ as  $r\rightarrow+\infty$,
\end{center}
 we obtain
\begin{center}
$L(y-u)=e^{2r^2}(14r^2+4m)+\lambda e^{-\frac{r^2}{2(m-2)}}u\rightarrow +\infty$ as $r\rightarrow+\infty$\end{center}
There exists $ R_0$ such that
\begin{center}
$ L(y-u)\ge 0$ \quad in  \quad $B_{R_0}^C.$
\end{center}
For fixed $R_0$, we can choose a constant $C_{R_0}$ such that 
\begin{center}
$y-C_{R_0}-u\le 0$ \quad on  \quad $\partial B_{R_0}.$
\end{center}
However,
\begin{center}
$ L((y-C_{R_0})-u)\ge 0$ \quad in  \quad $B_{R_0}^C.$
\end{center}
By the maximum principle, we obtain that
\begin{center}
$y-C_{R_0}-u\le 0$ \quad in \quad $B_{R_0}^C.$
\end{center}
So we obtain
\begin{center}
$u\ge y-C_{R_0} \to +\infty$ \quad as  \quad $r\to +\infty,$ 
\end{center}
which contradicts the assumption that $\lim\limits_{r\to +\infty}{ u(r,\cdot)=A}.$ Hence we have $u$ must be discontinuous at $\infty.$
\end{proof}
We can use the same method to prove the discontinuity of eigenfunctions of the drifted Laplace $\Delta_h=\Delta_{g_0}-\nabla_{g_0} h\cdot\nabla_{g_0} $.

\begin{proof}[Proof of Theorem \ref{thm:1.4}]
we know $u$ satisfies the following equation
 \begin{equation}
\Delta _h u= -\lambda u.
\end{equation}
We rewrite it in the following form
\begin{equation}
\Delta_{g_0} u-( \nabla_{g_0} h,\nabla_{g_0} u)= -\lambda u .
\end{equation}
We can do the similar computation as in the beginning of Section 3 
 \begin{align*}
\Delta_h&= \Delta_{g_0}-\nabla_{g_0} h\cdot\nabla_{g_0}\\
&=\frac{\partial^2}{\partial r^2}+ \frac{m-1}{r}\frac{\partial}{\partial r}+\frac{1}{r^2} \Delta_\theta-\frac{r}{2}\frac{\partial}{\partial r}.\\
\end{align*}
where $\Delta_\theta$ is the Laplacian on the standard $S^{m-1}.$
Then (3.5) can be written as 
\begin{equation}
u_{rr}+\frac{m-1}{r}u_r+\frac{1}{r^2}\Delta_\theta u-\frac{r}{2}\frac{\partial u}{\partial r}=-\lambda u.
\end{equation}
Let $\varphi_k$ be the orthonormal basis on $L^2(S^{m-1})$ corresponding to the eigenvalues \\
\[0=\lambda_0<\lambda_1\le \lambda_2\le \cdots \le \lambda_k\rightarrow \infty \]
and satisfying 
\[\Delta_\theta \varphi_k=\lambda_k \varphi_k.\]
Let $f_k(r)=\langle u(r,\cdot),\varphi_k\rangle$ for $k\ge1$ and we denote $f=f_k.$
Then from (3.5) we see that $f=f_k$ satisfies 
\begin{equation}
f_{rr}+(\frac{m-1}{r}-\frac{r}{2})f_r=(-\lambda +r^{-2}\lambda_k)f.
\end{equation}
We prove it by contradiction. Assume $u$ is continuous at $\infty, \lim\limits_{r\to +\infty} {u(r,\cdot)=A}$, then
\begin{center}
$f_0(+\infty)=\lim\limits_{r\to +\infty}{\langle u(r,\cdot),\varphi_0\rangle}=\langle \lim\limits_{r\to +\infty} {u(r,\cdot),\varphi_0\rangle}=\langle A,\varphi_0\rangle=c_0.$
\end{center}

On the other hand, $f_0$ satisfies following equation
\[f_{rr}+(\frac{m-1}{r}-\frac{r}{2})f_r=-\lambda f.\]
Let $L=\frac{d^2}{dr^2}+(\frac{m-1}{r}-\frac{r}{2})\frac{d}{dr}$ and we choose $y=e^{2r^2},$ then we have
\begin{center}
$L(y-f_0)=e^{2r^2}(14r^2+4m)+\lambda f _0\rightarrow +\infty$ as $r\rightarrow +\infty.$
\end{center}
 
There exists $\widetilde {R_0}$ such that
\begin{center}
$ L(y-f_0)\ge 0$ \quad in  \quad $B_{\widetilde{R_0}}^C.$
\end{center}
For fixed $R_0$, we can choose a constant $C_{\widetilde{R_0}}$ such that 
\begin{center}
$y-C_{\widetilde{R_0}}-f_0\le 0$ \quad on  \quad $\partial B_{\widetilde{R_0}}.$
\end{center}
However,
\begin{center}
$ L((y-C_{\widetilde{R_0}})-f_0)\ge 0$ \quad in  \quad $B_{\widetilde{R_0}}^C.$
\end{center}
By the maximum principle, we obtain that
\begin{center}
$y-C_{\widetilde{R_0}}-f_0\le 0$ \quad in \quad $B_{\widetilde{R_0}}^C.$
\end{center}
So we obtain
\begin{center}
$f_0 \ge y-C_{\widetilde{R_0}}\to +\infty$ \quad as  \quad $r\to +\infty,$ 
\end{center}
which contradicts the assumption that $\lim\limits_{r\to +\infty}{ u(r,\cdot)=A}.$ Hence we have $u$ must be discontinuous at $\infty.$
\end{proof}
\section{ existence of  solution of the equation $-\Delta_gu=f$ }
\begin{theorem}
Assume $f\in L^2(\bar M)$, then the following equation\\
\begin{equation}
 \left  \{
\begin{aligned}
&-\Delta_g u =f;\\
&u\in H^1_0(\bar{M})
\end{aligned}
\right.
\end{equation}
has a unique solution  $u\in H^1_0(\bar{M}).$
\end{theorem}
\begin{proof}
We rewrite the equation of (4.1)  in the distributional sense
 \begin{center}
$-(\Delta_g u,\varphi)_g=(f,\varphi)_g \quad$ for all $ \varphi \in H^1_0(\bar{M}).$
\end{center}
It follows that 
\begin{equation}
(\nabla_g u,\nabla_g\varphi)_g=(f,\varphi)_g.
\end{equation}
We set 
\begin{equation}
[u,\varphi]=(\nabla_g u,\nabla_g\varphi)_g( \forall u,\varphi \in H^1_0(\bar{M}) ).
\end{equation}
A logarithmic  Sobolev inequality (1.1) implies a Poincar$\acute{e}$ inequality (cf \cite{8} Prop 2.1)
\begin{align*}
\int_{\mathbb{R}^m}u^2dV_g&=\int_{\mathbb{R}^m}u^2e^{-\frac{m|x|^2}{m-2}}dV_{g_0}\\
&\le \frac{2(m-2)}{m}\int_{\mathbb{R}^m} |\nabla u|_{g_0}^2e^{-\frac{m|x|^2}{m-2}}dV_{g_0}\\
&\le \frac{2(m-2)}{m}\int_{\mathbb{R}^m} |\nabla u|_{g_0}^2e^{\frac{|x|^2}{2(m-2)}}e^{-\frac{m|x|^2}{m-2}}dV_{g_0}\\
&=\frac{2(m-2)}{m}\int_{\mathbb{R}^m} |\nabla u|_g^2dV_g.\\
\end{align*}
Therefore, the space $H^1_0(\bar M)$ with the inner product $[\cdot,\cdot]$ is complete. The equation $(4.1)$ can be written as  
\[[u,\varphi]_\alpha=(f,\varphi)_g.\]
On the other hand,
\[|(f,\varphi)_g|\le||f||_{L^2(\bar{M})}||\varphi||_{L^2(\bar{M})}\le ||f||_{L^2(\bar{M})}[\varphi,\varphi]^{\frac{1}{2}},\]
we know that $\varphi \mapsto \int_{\mathbb{R}^m}f \cdot \varphi dV_g \quad (\forall  \varphi \in H^1_0(\bar M))$ is a bounded linear function.\\
By the Riesz representation theorem, we have a unique solution $u\in H^1_0(\bar M)$. 
\end{proof}
~\\
\renewcommand{\abstractname}{Acknowledgements}
\begin{abstract}
I would like to show my deepest gratitude to my
supervisor, Prof. Li, Jiayu, who has provided me with valuable guidance in every stage of the writing of this
thesis.
\end{abstract}
 ~\\


\begin{thebibliography}{99}
\bibitem{1}Lin Fanghua, Wang Changyou, {\em Harmonic and quasi-harmonic spheres}, Comm. Anal. Geom.  {\bf 7}(1999), 397-429.
\bibitem{2}Ding, Weiyue, Zhao Yongqiang, {\em Elliptic equations strongly degenerate at a point}, Nonlinear Anal. {\bf 65}(2006), 1624-1632.
\bibitem{3}D Bakery and M $\acute{E}$mery, {\em Diffusions hypercontractives, Seminaire de probabilities,XIX,1983/84}, Lecture Notes in Math. {\bf 1123}(1985), 177-206.
\bibitem{4}
Xu Cheng and Detang Zhou, {\em Eigenvalues of the drifted Laplace on complete metric measure space}, arXiv: 1305.4116v2[math DG] 16 Oct 2013.
\bibitem{5}
Hans-Joachim Hein and Aaron Naber, {\em New logarithmic Soblev inequalities and a $\epsilon$-regularity theorem for the Ricci flow},  arXiv:1205.0380v1 [math DG] 2 May 2012.
\bibitem{6}
Alexander Grigoryan, Heat Kernel and Analysis on Manifolds, American Mathematical Soc., 2009(English)
\bibitem{7}
Jos$\acute{e}$F. Escobar, {\em On the spectrum of the Laplacian on complete riemannian manifolds}, Communications in Partial Differential Equations, {\bf 11:1}(1986), 63-85. 
\bibitem{8}
Michel Ledoux, {\em Concentration of measure and logarithmic Sobolev inequalities, S$\acute{e}$minaire de probabiliti$\acute{e}$s de Strasbourg} {\bf 33} (1999), 120-216.
\bibitem{9}
P.Mayer, {\em Probability and Potentials}, Blaisdell Publishing Co.,  (1991).
\bibitem{10}
Michael Struwe, {\em On the evolution of harmonic maps in higher dimensions},J.Differential Geometry {\bf 28} (1988), 485-502
\end{thebibliography}
\end{document}